\documentclass[11pt]{article}
\usepackage{amsfonts}
\usepackage{latexsym}
\usepackage{amsmath}
\usepackage{amssymb}
\usepackage{color}
\usepackage{amsthm}
\usepackage{fullpage}

\usepackage[T1]{fontenc}
\usepackage[utf8]{inputenc} 

\usepackage{enumerate}
\newtheorem{theorem}{Theorem}[section]
\newtheorem{lemma}[theorem]{Lemma}

\newtheorem{proposition}[theorem]{Proposition}

\newtheorem{question}[theorem]{Question}
\newtheorem{remark}[theorem]{Remark}

\theoremstyle{definition}
\newtheorem{definition}[theorem]{Definition}
\newtheorem*{example}{Example}

\newtheorem{thmy}{Theorem}

\newtheorem*{note*}{Note}

\makeatletter

\newcommand{\R}{\mathbb{R}}

\newcommand{\HH}{\mathcal H}

\makeatother

\makeatletter
\newcommand{\subjclass}[2][1991]{%
	\let\@oldtitle\@title%
	\gdef\@title{\@oldtitle\footnotetext{#1 \emph{Mathematics subject classification.} #2}}%
}
\newcommand{\keywords}[1]{%
	\let\@@oldtitle\@title%
	\gdef\@title{\@@oldtitle\footnotetext{\emph{Key words and phrases.} #1.}}%
}
\makeatother

\begin{document}
	
	\title{On convex  bodies with  rotationally symmetric planar projections}

	\author{Sergii Myroshnychenko, Dmitry Ryabogin, and Christos Saroglou}
	
	\subjclass[2010]{52A30, 52A20, 52A39}
	\keywords{Sections and projections of convex bodies, characterizations of balls}

	\date{\today}
	\maketitle
	\begin{abstract}{Let $n\ge 3$ and let \(K\subset\mathbb R^n\) be a convex body. For a two-dimensional linear subspace
	\(P\subset\mathbb R^n\), let \(K|P\) be the orthogonal projection of \(K\) onto \(P\).
	We prove that if for every two-dimensional subspace \(P\), the planar
	convex body \(K|P\) has \(q\)-fold rotational symmetry up to translation, then 
	for  \(q\ge4\), this forces \(K\) to be an Euclidean ball. The case
	\(q=3\) is exceptional: non-spherical examples exist.}
\end{abstract}

\section{Introduction}

A recurring theme in convex geometry and geometric tomography is
the problem of uniqueness: to what extent can a convex body be recovered from information about its sections or projections? Closely related
questions ask whether a body can be characterized by special geometric properties of all its sections or projections. Such problems have
a long history, going back to the work of Minkowski, Funk, Banach,
Alexandrov, Auerbach, Mazur, and many others. We refer the reader to the following works for a few representative examples, \cite{AMU}, \cite{Gr}, \cite{M}, \cite{K1}, \cite{Bur}, \cite{Sch2}, \cite{BHJM}, \cite{IMN},   \cite{Za}, and to the books \cite{G}, \cite{Sc}, for related results.

Let $n\ge 3$. 
In what follows, we fix an orthonormal basis in the $n$-dimensional Euclidean space $\mathbb{R}^n$ and we denote by $\langle x,y\rangle$ the scalar product of two vectors $x,\ y$, with respect to this orthonormal basis. We will use the notation $S^{n-1}$ for the unit sphere in $\mathbb{R}^n$, i.e. the set $\{x\in\mathbb{R}^n:|x|=1\}$, where $|x|=\sqrt{\langle x,x\rangle}$ is the length of $x$. 

Let $2\le j\le n-1$ and let $Gr(n,j)$ be the set of all $j$-dimensional subspaces of ${\mathbb R^n}$. Denote by $O(P,j)$ the group of orthogonal transformations acting on the $j$-dimensional subspace $P$ of ${\mathbb R^n}$ and let $G=G(P,j)$  be a  subgroup of $O(P,j)$. We say that $G$ is {\it complete} if for any ellipsoid ${\mathcal E}\subset P$ centered at the origin, condition $\varphi({\mathcal E})={\mathcal E}$ for all $\varphi\in G$ implies that  ${\mathcal E}$ is an Euclidean ball.

Let \(K|P\) be the orthogonal projection of \(K\) onto \(P\) and let $G$ be complete. We say that \(K|P\) has a  {\it complete symmetry} if  for every  $\varphi\in G$ there is a  translation $a\in P$ for which  $\varphi(K|P)+a=K|P$.
 In this paper we address the following question
\begin{question}\label{vot1}
	Let $n\ge 3$, let $2\le j\le n-1$ and let $K\subset {\mathbb R^n}$ be a convex body. If for all $P\in Gr(n,j)$  the projections $K|P$ have a complete symmetry, does it follow that $K$ is an Euclidean ball?
\end{question}

It was shown in \cite{MRS} that the answer to Question \ref{vot1} is affirmative, provided $K=-K$.
In this paper we answer Question \ref{vot1} in  the case $j=2$. We have
\begin{theorem}\label{cor-main}
	 Let $n\ge 3$ and let \(K\subset\mathbb R^n\) be a convex body. 
	Assume that  for every  \(P\in Gr(n,2)\), the planar
	convex body \(K|P\) has \(q\)-fold rotational symmetry for some  $q=q(P)\ge 4$,  i.e., 
	there exists a rotation $\varphi_q\in SO(P,2)$ by the angle $\frac{2\pi}{q}$,  and a vector $b=b_P\in P$ such that 
	\begin{equation}\label{eq:11}
		\varphi_q(K|P)+ b_P=K|P.
	\end{equation} 
	Then 
 \(K\) is  a Euclidean ball. On the other hand, there exist strictly convex non-symmetric  bodies of constant width such that  all  their $2$-dimensional projections have 
	$3$-fold rotational symmetries.
\end{theorem}

Theorem \ref{cor-main} is a direct consequence of the following statement.

\begin{theorem}\label{cor-Main}
	Let $n\ge 3$ and let $f:S^{n-1}\to\mathbb{R}$ be a continuous function. Assume that  the restriction  of $f$ to every $1$-dimensional equator has a $q$-fold rotational symmetry for some  $q=q(P)\ge 4$, i.e., for every $P\in Gr(n,2)$ there exists a rotation $\varphi_q\in SO(P,2)$ by the angle $\frac{2\pi}{q}$,  and a vector $b=b_P\in P$ such that 
	\begin{equation}\label{eq:1}
		f(\varphi_q(x))+\langle b, x\rangle=f(x)\qquad\forall x\in S^{n-1}\cap P.
		\end{equation} 
		Then there exists a constant $c\in{\mathbb R}$ and a vector $a\in{\mathbb R^n}$ such that  $f(x)=c+\langle a, x\rangle $ $\forall x\in S^{n-1}$.

		On the other hand, if $q=3$,  the function $f(x)=x_1x_2x_3$, $x\in S^{n-1}$,  satisfy (\ref{eq:1}).
\end{theorem}

The difficult part is related to the case $q$ {\it odd}, $q\ge 3$. The proof for  {\it even} $q\ge 4$  is known and  follows directly from (\cite{MRS}, Theorem 1.9). We give a different proof here.
The affirmative part of Theorem \ref{cor-Main} admits the following sharper formulation.

\begin{theorem}\label{FC1}
	Let $n\ge 3$ and  let   $f:S^{n-1}\to\mathbb{R}$ be a continuous function. Assume that for every great circle ${\mathcal C}\subset S^{n-1}$, the  Fourier coefficients of  the restriction  $f|_{\mathcal C}$ corresponding to frequencies $2$ and $3$ vanish, that is $\widehat{f|_{\mathcal C}}(\pm j)=0$, $j=2,3$.  
	Then there exists a constant $c\in{\mathbb R}$ and a vector $a\in{\mathbb R^n}$ such that  $f(x)=c+\langle a, x\rangle $ $\forall x\in S^{n-1}$.
\end{theorem}

We remark that,  after restricting $f$ to a great circle, the same Fourier mode of $f$ may
	receive contributions from infinitely many ambient harmonic degrees,
	 therefore one cannot directly compare circle Fourier coefficients.
	Instead, one first observes  that the $L^2(S^2)$-closure of the set 
	\begin{equation}\label{plM1}
	V_{2,3}=\{f\in C(S^{2}):\, \forall {\mathcal C}\quad \widehat{f|_{\mathcal C}}(\pm j)=0,\,\,j=2,3\}
	\end{equation}
	   is a
	closed rotation-invariant subspace of $L^2(S^2)$. Next,  one uses harmonic projections $\Pi_\ell $, $\ell=0,1,2,\dots$,  to the invariant subspaces of ambient spherical harmonics, to show that
	$	f \in V$ implies 
	$\Pi_\ell f \in V
$ $\forall \ell$; see Lemma \ref{l1OO1}.
	Only after this separation does one analyze restrictions to
	equators degree-by-degree.

	Thus, the   key idea is to show  that 
	the {\it irreducibility prevents inter-degree cancellations}.
	This is the representation-theoretic heart of the argument.
We prove that {\it the only closed \(SO(3)\)-invariant subspace of $L^2(S^2)$ compatible with the
		equatorial symmetry condition (\ref{eq:1}) is the space of constants and linear
		functions.}
	Lemma  \ref{l1O1} is  a key technical part of the argument. It states that every nonzero spherical harmonic of degree $\ell$  exhibits every frequency of the correct
	parity on some equator. In the representation theoretic language this is equivalent to the fact that
	the orbit of one equatorial Fourier coefficient spans the entire dual representation. 

The paper is structured as follows. In Sections 2 and 3 we introduce the notation and explain the main idea of the proof. In Section 4 we prove all necessary auxiliary statements. The main results are proved  in Sections 5 and 6.

\section{Notation and auxiliary statements}

Let ${\mathbb R}$, ${\mathbb C}$ and ${\mathbb Z}$ be the sets of real, complex numbers and  integers correspondingly. 
Let $n\ge 3$. We will denote  by $dx$ the volume element on $S^{n-1}$ or on equators of $S^{n-1}$. The Fourier coefficients of $f$ will be denoted by $\widehat{f}(n)$,
$$
\widehat{f}(m)=\frac{1}{2\pi}\int\limits_0^{2\pi}f(\theta)e^{-im\theta}d\theta,\qquad m\in {\mathbb Z}.
$$
The trigonometric polynomial of degree $l$ is the expression of the form
$$
P_\ell(\theta)=\sum\limits_{j=-\ell}^\ell a_je^{ij\theta}, \qquad a_j\in {\mathbb C},\quad \theta\in [0,2\pi],\quad \ell\in {\mathbb Z}, \,\,\ell\ge 0.
$$
The {\it real } trigonometric polynomial of degree $l$ is the expression of the form
$$
T_\ell(\theta)=\sum\limits_{j=0}^\ell (b_j\cos (j\theta)+c_j\sin(j\theta)), \qquad b_j, c_j\in {\mathbb R},\quad \theta\in [0,2\pi],\quad \ell\in {\mathbb Z}, \,\,\ell\ge 0.
$$

The notation \(C(S^{n-1})\) and \(L^2(S^{n-1})\) will be used for the space of continuous and square integrable functions on $S^{n-1}$.
We will repeatedly use the well-known facts about spherical harmonics and from the theory of representations of orthogonal groups. We refer the reader to \cite{BR}, \cite{GMS}, \cite{Sch2} and \cite{V} for the information.
We denote by
$\HH_\ell(\R^n)$
the space  of homogeneous harmonic polynomials of variables $x_1,\dots, x_n$
of degree \(\ell\) in \(\R^n\) and by $\HH_\ell(S^{n-1})$ the space of their restrictions to $S^{n-1}$.
The notation $SO(n)$ and $O(n)$ is for the special orthogonal and orthogonal groups in ${\mathbb R^n}$. 
We will denote  by  $\Delta_{S^{n-1}}$ the spherical Laplacian (\cite{Gro}, \S 1.2).
We will repeatedly use the spherical harmonic decomposition, 
\[
L^2(S^2)=\widehat{\bigoplus}_{\ell=0}^\infty \mathcal H_\ell.
\]
 Here \(\mathcal H_\ell\) is the space of degree
\(\ell\) spherical harmonics. It is the eigenspace of the Laplace--Beltrami operator
\(\Delta_{S^2}\) with eigenvalue
$-\ell(\ell+1)$. We will denote by 
$
\Pi_\ell:L^2(S^2)\to \mathcal H_\ell
$
 the orthogonal projection onto $H_\ell$.

Recall  how the condition on symmetries of projections could be recorded using the {\it support functions} of convex bodies.
Let \(h_K:S^{n-1}\to\mathbb R\) be the support function of \(K\),
$$
h_K(u)=\sup_{x\in K}\langle x,u\rangle.
$$
If \(P\subset\mathbb R^n\) is a two-dimensional subspace, then for \(u\in P\cap S^{n-1}\),
$
h_{K|P}(u)=h_K(u).
$
Indeed,
\[
h_{K|P}(u)
=
\sup_{y\in K|P}\langle y,u\rangle
=
\sup_{x\in K}\langle \textrm{proj}_P x,u\rangle
=
\sup_{x\in K}\langle x,u\rangle
=
h_K(u),
\]
because \(u\in P\).

\section{Idea of the proof}

Let $q\ge 3$ and let $K$ be a convex body in ${\mathbb R^3}$. 
Assume that \(K|P\) has \(q\)-fold rotational symmetry up to translation. In other words, we assume that there exists \(b_P\in P\) such that
\[
R_P(K|P-b_P)=K|P-b_P,
\]
where \(R_P\) is
the rotation of \(P\) by \(\frac{2\pi}{q}\). 
Equivalently, the support function of $K-b_P$
is invariant under \(R_P\).

If \(e_1,e_2\) is an  orthonormal basis in \(P\), parametrize its great circle $S^2\cap P$ by
\begin{equation}\label{lpM2}
	u(\theta)=e_1\cos\theta+e_2\sin\theta,\qquad \theta\in[0,2\pi].
\end{equation}
If \(c\in P\), then
\[
\langle c,u(\theta)\rangle
=
\langle c,e_1\rangle\cos\theta+\langle c,e_2\rangle\sin\theta.
\]
We see that translating a planar convex body changes its support function on the corresponding
circle only by Fourier modes \(m=\pm1\).
{\it  Consequently, if a projection is invariant under rotation by \(\frac{2\pi}{q}\) after some
translation, then the restriction of \(h_K\) to that circle may contain only the modes}
\begin{equation}\label{modes}
{\mathcal A}=\{0,\ \pm1,\ \pm q,\ \pm2q,\ \pm3q,\dots\}.
\end{equation}

The main observation is the following

\begin{lemma}[Parity of Fourier modes]\label{lem:parity}
Let \(P_\ell(x_1,x_2,x_3)\) be a homogeneous polynomial of degree \(\ell\).
Restrict it to a great circle
$
u(\theta)=e_1\cos\theta+e_2\sin\theta,
$ $\theta\in [0,2\pi]$.
Then \(P_\ell(u(\theta))\) is a real trigonometric polynomial involving only frequencies
$
\ell,\ \ell-2,\ \ell-4,\dots,
$
down to \(0\) if \(\ell\) is even and down to \(1\) if \(\ell\) is odd.
\end{lemma}

The key point of the proof of Theorem \ref{cor-main} is that 
one must not first restrict the entire support function $h_K$  to a great
circle and then compare Fourier coefficients. Indeed, if $h_K=\sum\limits_{\ell=0}^{\infty}H_\ell$ is the decomposition into ambient spherical harmonics on $S^2$, after restriction, the same circle
Fourier mode $\ell$ can receive contributions from infinitely many ambient harmonics  $H_j$ of degrees $j\ge \ell$. Hence, the cancellations of modes different from the ones in (\ref{modes}) might occur.

Instead, one first observes that the symmetry condition defines a closed rotation-invariant
subspace (\ref{plM1}) of \(L^2(S^2)\), with $f=h_K$. Then, using the theory of group representations of  $SO(3)$, one can show  that the spherical-harmonic projections $H_\ell$ of  $h_K$ 
  also satisfy (\ref{eq:1}) with $f=H_\ell$  for every $\ell$. Thus, the problem may be studied
degree by degree before restricting to circles.

\section{Auxiliary lemmata}

\subsection{Allowed circle modes}

Let $P\in Gr(n,2)$ and let \({\mathcal C}=S^{n-1}\cap P\) be a great circle. Choose  parametrization $u(\theta)$, $\theta\in [0,2\pi]$, as in (\ref{lpM2}). We show now  that  if the restriction of $f\in C(S^{n-1})$ onto ${\mathcal C}$ satisfies (\ref{eq:1}), then only modes in (\ref{modes}) participate in its Fourier series decomposition.
Indeed, computing the Fourier coefficients of both parts of (\ref{eq:1}), we have
\begin{equation}\label{coef}
\int\limits_0^{2\pi}f\Big(\theta+\frac{2\pi}{q}\Big)e^{-im\theta}d\theta+\int\limits_0^{2\pi}(b_1\cos\theta+b_2\sin\theta)e^{-im\theta}d\theta=\int\limits_0^{2\pi}f(\theta) e^{-im\theta}d\theta.
\end{equation}
Since
$$
b_1\cos\theta+b_2\sin\theta=c_1e^{-i\theta}+c_2e^{i\theta}
$$
for some constants $c_1$ and $c_2$, the second term in the left-hand side  of (\ref{coef})  is non-trivial, only provided $m= \pm 1$.

Let $m\neq \pm 1$. Then (\ref{coef}) reads as
$$
e^{\frac{2\pi i m}{q}}\,\int\limits_0^{2\pi}f(\theta) e^{-im\theta}d\theta=\int\limits_0^{2\pi}f(\theta) e^{-im\theta}d\theta,
$$
which is nothing but 
$$
\Big(e^{\frac{2\pi i m}{q}}-1\Big)\widehat{f}(m)=0\qquad \forall m\in {\mathbb Z}\setminus\{\pm1\}.
$$
Then $\widehat{f}(m)=0$, unless $m\in {\mathcal A}$, where ${\mathcal A}$ was defined in (\ref{modes}).

Conversely, let the series 
$$
f(\theta)=\sum\limits_{m\in{\mathcal A}}\widehat{f}(m)e^{im\theta} \qquad \forall \theta\in [0,2\pi],
$$
be convergent in the $L^2$-sense, and let 
 ${\mathcal A}$ be  defined  by (\ref{modes}) for some $q\ge 3$. Then $f$ satisfies (\ref{eq:1}). 

\begin{definition}\label{Adm1}
	Let \(V_{{\mathcal A}^c}\subset C(S^2)\) be defined as  the space of all functions \(f\) such that, for every great
	circle \({\mathcal C}\), the restriction \(f|_{\mathcal C}\) has no Fourier modes outside
	$\mathcal A$.
	Equivalently, for every great circle \({\mathcal C}\) and every integer \(m\notin\mathcal A\),
	$\widehat{f|_{\mathcal C}}(m)=0$.
\end{definition}

\begin{example}
	Let $q=3$. For support functions, the condition \(f\in V\) expresses the fact that every planar
	projection has three-fold rotational symmetry up to translation. On each projection plane,
	the translation accounts for the first Fourier modes, and the $3$-fold symmetry accounts
	for the modes divisible by \(3\).
\end{example}

\begin{definition}\label{Adm2}
Let  \(V_{2,3}\subset C(S^2)\) be the space of all functions \(f\) such that, 
the  Fourier coefficients of  the restriction  $f|_{\mathcal C}$ corresponding to frequencies $2$ and $3$ vanish, that is $\widehat{f|_{\mathcal C}}(\pm j)=0$, $j=2,3$.  
\end{definition}

	Observe  that $V_{{\mathcal A}^c}\subset V_{2,3}$, provided $q\ge 4$.

\subsection{The cancellation problem}

Let $n=3$ and let 
\[
f=\sum_{\ell=0}^\infty H_\ell,\qquad H_\ell\in\mathcal H_\ell(S^2),
\]
be  the decomposition of $f$ into spherical harmonics, \cite{Sch2}.
Then after restriction to a great circle, the same Fourier mode can come from several
different \(H_\ell\)'s. For example, the mode \(5\) on a great circle can receive
contributions from degrees \(5,7,9,\dots\).
Hence,   equality
$\widehat{f|_C}(5)=0$
does not by itself imply that the degree \(5\) component $H_5$ vanishes. There could be
cancellations among degrees.
The way to avoid this problem is to prove first that
condition 
$f\in V$ yields $ H_\ell\in V$ for every $\ell$.
Once this is known, we can study each degree separately, and no inter-degree cancellation
is possible.

\subsection{Closedness and rotation invariance of \(V_{{\mathcal A}^c} \) and \(V_{2,3}\)}

Let $f\in C(S^2)$.
Fix the equator
\[
{\mathcal C}_0=\{(\cos\theta,\sin\theta,0):0\le\theta<2\pi\}.
\]
For \(Q\in SO(3)\), let
\[
{\mathcal C}_Q=Q {\mathcal C}_0.
\]
For a continuous function \(f\), define
\[
\Lambda_{m,Q}(f)
=
\frac1{2\pi}\int_0^{2\pi}
f\bigl(Q(\cos\theta,\sin\theta,0)\bigr)e^{-im\theta}\,d\theta,\qquad m\in{\mathbb Z}.
\]
This is the \(m\)-th Fourier coefficient of the restriction of \(f\) to the great circle
\({\mathcal C}_Q\).

\begin{lemma}
Let $q\ge 3$ and let ${\mathcal A}$ be defined by  (\ref{modes}). The spaces \(V_{{\mathcal A}^c} \) and \(V_{2,3}\) are linear  and invariant under rotations.
\end{lemma}

\begin{proof} We give a proof for \(V=V_{{\mathcal A}^c} \). The proof for \(V_{2,3}\) is analogous.

	We show the rotation invariance first.
	Let \(f\in V\), and let \(Q_0\in SO(3)\). Define
	\[
	(Q_0 f)(u)=f(Q_0^{-1}u).
	\]
	Let \({\mathcal C}\) be a great circle. Then \(Q_0^{-1}{\mathcal C}\) is also a great circle. The restriction
	of \(Q_0f\) to \({\mathcal C}\) is the restriction of \(f\) to \(Q_0^{-1}{\mathcal C}\), possibly with a shift
	of the angular parameter. A shift of the angular parameter multiplies the \(m\)-th Fourier
	coefficient by \(e^{im\alpha}\), $\alpha\in [0,2\pi]$, and hence cannot create new Fourier modes. Therefore
	\(Q_0f\) has only allowed modes on every great circle, so \(Q_0f\in V\).
	
	Now, let $f,g\in V$ and let \({\mathcal C}\) be a great circle. Then the restriction of $f, g$ to \({\mathcal C}\) have only allowed modes. Then the restriction of $f+g$ to \({\mathcal C}\) has only allowed modes. Since \({\mathcal C}\) is arbitrary, $f+g\in V$. The proof  showing that $\lambda f\in V$ for any $\lambda\in {\mathbb C}$, provided $f\in V$, is similar.
	\end{proof}

\begin{lemma}\label{dlp4}
	Let $q\ge 3$ and let ${\mathcal A}$ be defined by  (\ref{modes}). The spaces \(V_{{\mathcal A}^c} \) and \(V_{2,3}\) are closed in the topology of uniform convergence on \(S^2\).
\end{lemma}

\begin{proof}
	Let $V=V_{{\mathcal A}^c} $ or $V=V_{2,3}$, and let 
	\(f_j\in V\). Assume \(f_j\to f\) uniformly on \(S^2\). Fix a forbidden mode
	\(m\notin\mathcal A\) and a rotation \(Q\in SO(3)\). Then
	$
	\Lambda_{m,Q}(f_j)=0
	$
	for every \(j\). Since
	\[
	|\Lambda_{m,Q}(f_j)-\Lambda_{m,Q}(f)|
	\le
	\sup_{S^2}|f_j-f|,
	\]
	we obtain
	$
	\Lambda_{m,Q}(f)=0.
	$
	Since \(m\notin\mathcal A\) and \(Q\in SO(3)\) were arbitrary, \(f\in V\).
\end{proof}

\begin{definition}
	We define
	\[
	{\mathcal V}_{{\mathcal A}^c}=\overline{V_{{\mathcal A}^c}}^{\,L^2(S^2)}, \qquad 	{\mathcal V}_{2,3}=\overline{V_{2,3}}^{\,L^2(S^2)}
	\]
	Thus, these spaces are  the \(L^2\)-closed subspaces generated by the continuous functions
	satisfying the corresponding great-circle Fourier condition.
\end{definition}

\begin{lemma}\label{votjetoda1}
	The spaces $	{\mathcal V}_{{\mathcal A}^c},	{\mathcal V}_{2,3}$ are  closed linear subspaces that are invariant under
	the natural action of \(SO(3)\).
\end{lemma}

\begin{proof}
	Linearity is immediate because  \(V_{{\mathcal A}^c} \) and \(V_{2,3}\) are  linear, hence their  \(L^2\)-closures are  linear.
	Closedness is true by definition.
	
	It remains only to check rotation invariance. We will check it for ${\mathcal V}_{2,3}$, the proof of rotation invariance of  $	{\mathcal V}_{{\mathcal A}^c}$ is similar.
	For \(R\in SO(3)\), define
	$
	T_R f(u):=f(R^{-1}u)$, $u\in S^2$.
	The operator \(T_R\) is unitary on \(L^2(S^2)\), because the spherical measure is rotation invariant. In particular, \(T_R\) is continuous in the \(L^2\)-norm.
	
	Let \(f\in {\mathcal V}_{2,3}\). By definition of \({\mathcal V}_{2,3}\), there exists a sequence \(f_j\in V_{2,3}\) such that
	$
	f_j\to f$ in $L^2(S^2)$.
	Applying \(T_R\), and using the continuity of \(T_R\), we get
	$
	T_R f_j\to T_R f$ in $L^2(S^2)$.
	By the previous lemma \(V_{2,3}\) is rotation invariant, so \(T_R f_j\in V_{2,3}\) for every \(j\). Therefore \(T_R f\) belongs to the \(L^2\)-closure of \(V_{2,3}\), that is,
	$T_R f\in {\mathcal V}_{2,3}$.
	Hence,
	$
	T_R {\mathcal V}_{2,3}\subset {\mathcal V}_{2,3}.
	$
	Applying the same argument to \(R^{-1}\), we also obtain
	$
	T_{R^{-1}}{\mathcal V}_{2,3}\subset {\mathcal V}_{2,3}.
	$
	This implies the reverse inclusion \({\mathcal V}_{2,3}\subset T_R{\mathcal V}_{2,3}\). Consequently,
	$
	T_R{\mathcal V}_{2,3}={\mathcal V}_{2,3}.
	$
	Thus \({\mathcal V}_{2,3}\) is \(SO(3)\)-invariant.
\end{proof}

\begin{remark}
	For a general function in
	\(L^2(S^2)\), the restriction to a great circle is not intrinsically defined, since a
	great circle has spherical measure zero. Therefore the Fourier conditions on great
	circles are imposed first on continuous functions, and only then is the corresponding
	space closed in \(L^2(S^2)\). This is the natural setting for applying the spherical
	harmonic decomposition.
\end{remark}

\subsection{Spherical harmonics and spectral projections}

The following lemma is the central point, (cf. \cite{V}, Chapter I,  \S 3, section 3). It shows that harmonic projections preserve closed rotation-invariant subspaces. Recall that
$
\Pi_\ell:L^2(S^2)\to \mathcal H_\ell
$
stands for  the orthogonal projection on $\HH_\ell$, the space of spherical harmonics of degree $\ell$, $\ell=0,1,\dots$

\begin{lemma}\label{l1OO1}
Let \(W\subset L^2(S^2)\) be a closed linear subspace invariant under the natural action of
\(SO(3)\). Then
\[
f\in W \quad\Longrightarrow\quad \Pi_\ell f\in W
\]
for every \(\ell\ge0\).
\end{lemma}

\begin{proof}
For \(Q\in SO(3)\), let
\[
(U_Qf)(u)=f(Q^{-1}u),\qquad u\in S^2.
\]
By assumption, \(U_QW=W\) for every \(Q\in SO(3)\).

Recall the standard character formula
\[
\Pi_\ell
=(2\ell+1)\int_{SO(3)}\chi_\ell(Q^{-1})U_Q\,dQ,
\]
where \(\chi_\ell\) is the character of the irreducible representation
\(\mathcal H_\ell(S^2)\), and \(dQ\) is normalized Haar measure; see, for example (\cite{BR}, Chapter 7, \S 3, formula (8)).

Let \(f\in W\). Since \(U_Qf\in W\) for every \(Q\in SO(3)\), every finite
linear combination of such vectors belongs to \(W\). The above Bochner integral is the
limit in \(L^2(S^2)\) of finite sums of this form. Since \(W\) is closed, it follows that
\(\Pi_\ell f\in W\).
\end{proof}

\begin{remark}
	This lemma is the rigorous reason why cancellations among different harmonic degrees are
	irrelevant. If \(f\in V\), then each harmonic component \(\Pi_\ell f\) is also in \(V\).
	Only after this degree separation do we restrict to great circles and look at Fourier modes.
\end{remark}

\begin{lemma}\label{harmoniccomponent}
	Let
	$
	{\mathcal V}_{2,3}$ be the closure of
	$V_{2,3}
	$ in $L^2(S^2)$
	and let
	$
	f=\sum_{\ell=0}^{\infty}H_\ell$,
	$
	H_\ell\in\mathcal H_\ell(S^2)$,
	be the spherical harmonic expansion of a function
	\(f\in{\mathcal V}_{2,3}\).
	Then
	$
	H_\ell\in V_{2,3}
	$
	for every \(\ell\ge0\).
\end{lemma}

\begin{proof}
Fix \(\ell\ge0\). By Lemmas~\ref{votjetoda1} and~\ref{l1OO1},
\[
H_\ell=\Pi_\ell f\in\mathcal V_{2,3}.
\]
Hence there exists a sequence \(f_j\in V_{2,3}\) such that
\[
f_j\longrightarrow H_\ell
\qquad\text{in }L^2(S^2).
\]
Since \(\Pi_\ell\) is bounded,
\[
\Pi_\ell f_j\longrightarrow H_\ell
\qquad\text{in }L^2(S^2).
\]

We claim that \(\Pi_\ell f_j\in V_{2,3}\) for every \(j\). Indeed, by the character
formula,
\[
\Pi_\ell f_j
=(2\ell+1)\int_{SO(3)}\chi_\ell(Q^{-1})U_Qf_j\,dQ.
\]
For every \(Q\in SO(3)\), rotation invariance gives \(U_Qf_j\in V_{2,3}\).
Moreover, the map
\[
Q\longmapsto \chi_\ell(Q^{-1})U_Qf_j
\]
is continuous from \(SO(3)\) to \(C(S^2)\), equipped with the uniform norm. Thus the
integral is the uniform limit of Riemann sums. Every such Riemann sum belongs to
\(V_{2,3}\), because this space is linear and rotation invariant. Since \(V_{2,3}\) is
closed under uniform convergence by Lemma~\ref{dlp4}, we obtain
\(\Pi_\ell f_j\in V_{2,3}\).

Finally, \(\Pi_\ell f_j\) and \(H_\ell\) belong to the finite-dimensional space
\(\mathcal H_\ell(S^2)\). Hence their \(L^2\)-convergence is also uniform. Applying
Lemma~\ref{dlp4} once more yields \(H_\ell\in V_{2,3}\).
\end{proof}

\subsection{Proof of Lemma \ref{lem:parity}}

\begin{proof}
	Write
	$
	e_1=(a_1,a_2,a_3)$, $e_2=(b_1,b_2,b_3)$.
	Then
	$u_i(\theta)=a_i\cos\theta+b_i\sin\theta$, $\theta\in [0,2\pi]$.
	Since \(P_\ell\) is homogeneous of degree \(\ell\), the expression \(P_\ell(u(\theta))\)
	is a homogeneous degree-\(\ell\) polynomial in \(\cos\theta\) and \(\sin\theta\). It is a
	linear combination of monomials
	$
	\cos^r\theta\,\sin^{\ell-r}\theta$, $ 0\le r\le \ell$.
	Using
	\[
	\cos\theta=\frac{e^{i\theta}+e^{-i\theta}}2,\qquad
	\sin\theta=\frac{e^{i\theta}-e^{-i\theta}}{2i},
	\]
	we see that each monomial is a linear combination of exponentials 
	$$
	e^{i(j+k)\theta}e^{-i(\ell-(j+k))\theta}=e^{-i(\ell-2(j+k))\theta}, 
	$$
	where $j=0,1,\dots, r$ and $k=0,1,\dots, \ell-r$,
 and
	\(|k+j|\le \ell\). Therefore, writing back everything in terms of $\sin$ and $\cos$ we see that  the only possible frequencies are
	$\ell,\ell-2,\ell-4,\dots,0$, provided $\ell$ is even and $\ell,\ell-2,\ell-4,\dots,1$, provided $\ell$ is odd.
\end{proof}

\subsection{Restriction of standard spherical harmonics to the equator}

Let $\ell=0,1,2,\dots$, let $\varphi \in [0,2\pi]$ and let $\vartheta \in [0,\pi]$. Recall that the complex spherical harmonic of degree \(\ell\) and order
\(r\)  has the form
\[
Y_\ell^r(\vartheta,\varphi)
=
N_{\ell r}\,
P_\ell^{\,r}(\cos\vartheta)e^{ir\varphi},\qquad r=-\ell,-\ell+1,\dots,\ell-1,\ell,
\]
where \(N_{\ell r}\neq0\) is a normalization constant,
	 \(P_\ell^{\,r}\) is the associated Legendre function,
 \(\vartheta\) is the polar angle,
and  \(\varphi\) is the azimuthal angle, (\cite{GMS}, Chapter I, \S3, section 4).

\begin{lemma}\label{votono1}
The restriction of $Y_\ell^r$ to  the equator ${\mathcal C}_0=(\cos\varphi,\sin\varphi,0)$, $\varphi\in[0,2\pi]$, has a nonzero \(r\)-th Fourier coefficient  whenever
$0\le |r|\le \ell,$ $r \equiv \ell \pmod 2$.
\end{lemma}
\begin{proof}
On the equator we have
$\vartheta=\frac{\pi}{2}$,
$\cos\vartheta=0$.
Therefore,
$Y_\ell^r|_{C_0}
=
N_{\ell r}\,
P_\ell^{\,r}(0)\,
e^{ir\theta}$.
Hence the \(r\)-th Fourier coefficient of the restriction to the
equator is nonzero exactly when
$P_\ell^{\,r}(0)\neq0$.
If \(\ell-r\) is odd, then 
	$P_\ell^{\,r}(0)=0$.
	On the other hand, if \(\ell-r\) is even, then \(P_\ell^{\,r}\) is even, and in fact
	$P_\ell^{\,r}(0)\neq0$, (see \cite{Gro}, Lemma 3.3.8, or \cite{Sch2}, Lemma 2.6).
Consequently,
$
P_\ell^{\,r}(0)\neq0$ if and only if 
$
\ell-r $ is even.
Equivalently,
$r\equiv \ell \pmod 2$.
Therefore the restriction of \(Y_\ell^r\) to the equator contains a
nonzero Fourier mode \(e^{ir\theta}\) precisely when
$r\equiv \ell \pmod2$.
\end{proof}

\subsection{A nonzero harmonic exhibits any admissible frequency on some equator}

\begin{lemma}\label{l1O1}
	Let \(Y\in\mathcal H_\ell(S^2)\), \(Y\ne0\). Let \(r\) satisfy
	$0\le |r|\le \ell,$ $ r\equiv\ell\pmod2$.
	Then there exists a great circle \({\mathcal C}\) such that the restriction \(Y|_{\mathcal C}\) has a nonzero
	\(r\)-th Fourier coefficient.
\end{lemma}

\begin{proof}
	Fix the equator \(C_0\), and define the linear functional
	\[
	\Lambda_r(Y)
	=
	\frac1{2\pi}\int_0^{2\pi}
	Y(\cos\theta,\sin\theta,0)e^{-ir\theta}\,d\theta.
	\]
	This functional is not identically zero on \(\mathcal H_\ell\). Indeed, by the previous lemma
	spherical harmonic \(Y=Y_\ell^r\) has a nonzero \(r\)-th Fourier coefficient on the equator
	whenever \(r\equiv\ell\pmod2\).
	
	For a rotation \(Q\in SO(3)\), define
	\[
	\Lambda_{r,Q}(Y)=\Lambda_r(Y\circ Q^{-1}).
	\]
	This extracts the \(r\)-th Fourier coefficient of \(Y\) on the rotated great circle
	\(Q C_0\).
	
	Suppose, for contradiction, that \(Y\) has zero \(r\)-th Fourier coefficient on every
	great circle. Then
	\[
	\Lambda_{r,Q}(Y)=0
	\qquad\forall Q\in SO(3).
	\]
	Equivalently, \(Y\) is annihilated by every element in the \(SO(3)\)-orbit of the nonzero
	functional \(\Lambda_r\).
	
	Consider now the dual space \(\mathcal H_\ell^*=\{L:\mathcal H_\ell\to{\mathbb C} \,\,\textrm{linear}\}\). Let $Q\to \rho(Q)$ be a representation of $SO(3)$ on $\mathcal H_\ell$, defined as $\rho(Q)Y(u)=Y(Q^{-1}u)$ for every $u\in S^2$.
	Observe that rotations act on functionals by
	$$
	(\rho^*(Q)L)(Y)=L(\rho(Q^{-1})Y),
	$$
	or  $(Q\cdot L)(Y)=L(Y\circ Q)$, (cf. \cite{V}, Chapter I, \S 1.4) This is  a natural induced action since 
	$$
	(\rho^*(Q)L)(\rho(Q)Y)=L(\rho(Q^{-1})\rho(Q)Y)=L(Y).
	$$
	Also, the above $\Lambda_{r,Q}(Y)=\Lambda_r(Y\circ Q^{-1})$ is  exactly $\Lambda_{r,Q}=\rho^*(Q^{-1})\Lambda_r$.  
	
	Consider the set 
$E=span({\mathcal O})$,
where ${\mathcal O}=\{\Lambda_{r,Q}:\,Q\in SO(3)\}\}$ is the orbit of one functional $\Lambda_r$  in the dual representation. 
Observe  that $E$ is an invariant subspace of \(\mathcal H_\ell^*\). Indeed, let $N\ge 1$, let $L=\sum\limits_{j=1}^Nc_j\Lambda_{r,Q_j}\in E$ and let $Q_0\in SO(3)$. Then,
$$
\rho^*(Q_0)L=\sum\limits_{j=1}^Nc_j\rho^*(Q_0)\Lambda_{r,Q_j}.
$$
	But $\rho^*(Q_0)\Lambda_{r,Q_j}=\Lambda_{r,Q_0Q}$, 
	$$
	(\rho^*(Q_0)\Lambda_{r,Q_j})(Y)=\Lambda_{r,Q_j}(\rho(Q_0^{-1})Y)=\Lambda_r(Y\circ Q_0^{-1}\circ Q_j^{-1})=\Lambda_{r,Q_0Q_j}(Y),
	$$
	and since $Q_0Q_j\in SO(3)$ every term is again in the orbit, so $\rho^*(Q_0)L\in E$ and $E$ is an invariant subspace of the dual representation.
	
	Now we claim that $\rho^*$  is irreducible on $\mathcal H_\ell^*$. 
	Indeed, let $E\subset \mathcal H_\ell^*$ be a nonzero invariant subspace. 
	We have to show that $E=\{0\}$ or $E=\mathcal H_\ell^*$.
	Define the  annihilator of $E$ as
	$$
	E^{\perp}=\{ v\in \mathcal H_\ell:\,L(v)=0\quad\forall L\in E\},
	$$
	which is a subspace of $\mathcal H_\ell$. We claim that $E^{\perp}$ is invariant under the original representation $\rho(Q)$, $Q\in SO(3)$. Take $v\in E^{\perp}$, $Q\in SO(3)$. We have to show that $\rho(Q)v\in E^{\perp}$. Let $L\in E$. Since $E$ is invariant, $\rho^*(Q^{-1})L\in E$ and because $v\in E^{\perp}$, $(\rho^*(Q^{-1})L)(v)=0$. But $(\rho^*(Q^{-1})L)(v)=L(\rho(Q)v)$, hence $L(\rho(Q)v)=0$ for all $L\in E$, so $\rho(Q)v\in E^{\perp}$. 
	
	Since $\mathcal H_\ell$ is irreducible, $E^{\perp}=\{0\}$ or $E^{\perp}=\mathcal H_\ell$. If 
	$E^{\perp}=\mathcal H_\ell$, then every $L\in E$ vanishes identically, so $E=\{0\}$.
	If 
	$E^{\perp}=\{0\}$, then  $E$ separates points of $\mathcal H_\ell$. Hence, $E=\mathcal H_\ell^*$.
	
	Thus, the span of the orbit of any non-zero functional $\Lambda_r$ is all of $\mathcal H_\ell^*$. We see that every linear functional on $\mathcal H_\ell(S^2)$ annihilates $Y$, so $Y=0$, a contradiction.
\end{proof}

\section{Proofs of Theorems \ref{cor-Main} and \ref{FC1}}

We first prove Theorem~\ref{FC1}. The essential argument is three-dimensional.
Let \(E\subset\mathbb R^n\) be a three-dimensional subspace. The restriction of \(f\)
to \(S^{n-1}\cap E\) satisfies the hypotheses of Theorem~\ref{FC1}; identifying \(E\)
with \(\mathbb R^3\), it is therefore enough to prove that
\begin{equation}\label{linear}
f(x)=c_E+\langle a_E,x\rangle,
\qquad x\in S^{n-1}\cap E,
\end{equation}
for some \(c_E\in\mathbb R\) and \(a_E\in E\).

Fix such an \(E\), and write the spherical harmonic expansion on \(S^{n-1}\cap E\cong
S^2\) as
\[
f=\sum_{\ell=0}^{\infty}H_\ell,
\qquad H_\ell\in\mathcal H_\ell(S^2).
\]
Since \(f\in V_{2,3}\), Lemma~\ref{harmoniccomponent} gives
\(H_\ell\in V_{2,3}\) for every \(\ell\ge0\). We claim that \(H_\ell=0\) whenever
\(\ell\ge2\).

If \(\ell\ge2\) is even, then frequency \(2\) has the correct parity and satisfies
\(2\le\ell\). By Lemma~\ref{l1O1}, every nonzero element of
\(\mathcal H_\ell(S^2)\) has a nonzero second Fourier coefficient on some great circle.
This is impossible for \(H_\ell\in V_{2,3}\). Thus \(H_\ell=0\).

If \(\ell\ge3\) is odd, the same argument, now with frequency \(3\), gives
\(H_\ell=0\). Consequently,
\[
f=H_0+H_1
\qquad\text{on }S^{n-1}\cap E,
\]
which is precisely \eqref{linear}.

It remains to pass from the three-dimensional restrictions to a single affine function
on \(S^{n-1}\). For \(u\in S^{n-1}\cap E\), equation \eqref{linear} gives
\[
c_E=\frac{f(u)+f(-u)}2.
\]
Let \(E_1,E_2\) be two three-dimensional subspaces. Choose
\(u_i\in S^{n-1}\cap E_i\), \(i=1,2\), and let \(E_3\) be a three-dimensional
subspace containing \(u_1\) and \(u_2\). The preceding identity yields
\(c_{E_1}=c_{E_3}=c_{E_2}\). Hence \(c_E\) is independent of \(E\); denote its
common value by \(c\).

Set \(g=f-c\), and define the positively homogeneous extension \(G:\mathbb R^n\to
\mathbb R\) by
\[
G(0)=0,
\qquad
G(x)=|x|g\!\left(\frac{x}{|x|}\right),\quad x\ne0.
\]
On every three-dimensional subspace \(E\), the function \(G|_E\) is linear by
\eqref{linear}. Any two vectors \(x,y\in\mathbb R^n\) lie in a common
three-dimensional subspace. Therefore
\[
G(x+y)=G(x)+G(y)
\quad\text{and}\quad
G(tx)=tG(x)
\]
for all \(x,y\in\mathbb R^n\) and \(t\in\mathbb R\). Thus \(G\) is a linear
functional: \(G(x)=\langle a,x\rangle\) for some \(a\in\mathbb R^n\). Restricting to
the sphere gives
\[
f(x)=c+\langle a,x\rangle,
\qquad x\in S^{n-1}.
\]
This proves Theorem~\ref{FC1}.

Now suppose that \(f\) satisfies the hypotheses of Theorem~\ref{cor-Main} for some $q=q(P)$,
\(q\ge4\). By the Fourier calculation in Section~4.1, the restriction of \(f\) to every
great circle has no modes outside
\[
\mathcal A=\{0,\pm1,\pm q,\pm2q,\ldots\}.
\]
In particular, its Fourier coefficients of orders \(\pm2\) and \(\pm3\) vanish. Hence
Theorem~\ref{FC1} applies and proves the affirmative part of
Theorem~\ref{cor-Main}.

\subsection{The exceptional triangular case}

The case \(q=3\) is different.

\begin{proposition}\label{prop:cubic}
Let
\[
H(u)=u_1u_2u_3,\qquad u\in S^2.
\]
Then for every great circle \({\mathcal C}\), the restriction \(H|_{\mathcal C}\) has only Fourier modes \(1\)
and \(3\).
\end{proposition}

\begin{proof}
Let \({\mathcal C}=P\cap S^2\), and choose an orthonormal basis \(e_1,e_2\) of \(P\). Write  
$$
u(\theta)=e_1\cos\theta+e_2\sin\theta
$$
 and let 
$e_1=(a_1,a_2,a_3)$, $e_2=(b_1,b_2,b_3)$.
Then
$u_i(\theta)=a_i\cos\theta+b_i\sin\theta$.
Therefore,
\[
H(u(\theta))
=
\prod_{i=1}^{3}(a_i\cos\theta+b_i\sin\theta).
\]
This is a homogeneous cubic polynomial in \(\cos\theta,\sin\theta\). Hence it is a linear
combination of
\[
\cos^3\theta,\quad
\cos^2\theta\sin\theta,\quad
\cos\theta\sin^2\theta,\quad
\sin^3\theta.
\]
Using
\[
\cos^3\theta=\frac{3\cos\theta+\cos3\theta}{4},\qquad
\sin^3\theta=\frac{3\sin\theta-\sin3\theta}{4},
\]
\[
\cos^2\theta\sin\theta=\frac{\sin\theta+\sin3\theta}{4},\qquad 
\cos\theta\sin^2\theta=\frac{\cos\theta-\cos3\theta}{4},
\]
we get
\[
H(u(\theta))
=
A\cos\theta+B\sin\theta+C\cos3\theta+D\sin3\theta
\]
for some $A,B, C, D\in{\mathbb R}$.
Thus, only modes \(1\) and \(3\) occur and the result follows.
\end{proof}

\begin{remark}
Since $C\cos3\theta+D\sin3\theta$ can be written as $a\cos3(\theta-\theta_0)$ for suitable $a\in\mathbb R$ and $\theta_0\in[0,2\pi]$, the function  \(\cos 3(\theta-\theta_0)\) is invariant under
$
\theta
\mapsto
\theta+\frac{2\pi}{3},
$
and also under reflection
$
\theta
\mapsto
2\theta_0-\theta .
$
Thus, it possesses the full dihedral symmetry group \(D_3\) of a regular
triangle.
Therefore every restriction of \(u_1u_2u_3\) onto any ${\mathcal C}$ becomes
\(D_3\)-symmetric after a suitable translation.
\end{remark}

\begin{remark}
	The  result similar to the one in Proposition holds for every $n\ge 3$, namely $H(u)=u_1u_2u_3$, $u\in S^{n-1}$ satisfies the conditions  for $q=3$.
\end{remark}

\subsection{Proof of Theorem \ref{cor-main}}

The assertion for $q\ge4$ follows from Theorem~\ref{cor-Main} by applying it to the support function $f=h_K$.

\begin{theorem}[Non-spherical examples for \(q=3\)]
For all sufficiently small \(\varepsilon>0\), the function
\[
h_\varepsilon(u)=1+\varepsilon u_1u_2u_3
\]
is the support function of a smooth strictly convex non-spherical body \(K_\varepsilon\), which is of constant width $2$, 
and every planar projection of \(K_\varepsilon\) has \(3\)-fold rotational symmetry up
to translation. 
\end{theorem}

\begin{proof}
The constant function \(1\) is the support function of the unit ball, and
\[
\nabla^2_{S^{n-1}}1+1\cdot I=I.
\]
Therefore, for sufficiently small \(|\varepsilon|\),
\[
\nabla^2_{S^{n-1}}h_\varepsilon+h_\varepsilon I
\]
is positive definite. Hence, \(h_\varepsilon\) is the support function of a smooth strictly
convex body, (\cite{G}, page 24-25). It is not centrally symmetric because \(u_1u_2u_3\) is odd.

Fix a plane \(P\). By Proposition~\ref{prop:cubic} and the same calculation in higher dimensions, on \(P\cap S^{n-1}\),
\[
u_1u_2u_3=A\cos\theta+B\sin\theta+C\cos3\theta+D\sin3\theta.
\]
The first harmonic \(A\cos\theta+B\sin\theta\) is exactly the effect of translating the
projection in \(P\). After removing this translation, the support function becomes
$
1+\varepsilon(C\cos3\theta+D\sin3\theta),
$
which is invariant under
$
\theta\mapsto\theta+\frac{2\pi}{3}.
$
Thus, each planar projection has \(3\)-fold rotational symmetry up to translation.
\end{proof}

\vspace{0.8 cm}

\noindent Sergii Myroshnychenko\\
Department of Mathematics and Statistics\\
University of the Fraser Valley\\
Abbotsford, BC V2S 7M8, Canada\\
E-mail address: serhii.myroshnychenko@ufv.ca

\vspace{0.5 cm}

\noindent Dmitry Ryabogin \\
Department of Mathematical Sciences\\
Kent State University\\
Kent, OH 44242, USA \\
E-mail address: ryabogin@math.kent.edu

\vspace{0.5 cm}

\noindent Christos Saroglou \\
Department of Mathematics \\
University of Ioannina\\
Panepistimioupoli, Ioannina 45110, Greece\\
E-mail address: csaroglou@uoi.gr \ \&\ christos.saroglou@gmail.com



\begin{thebibliography}{999}
	
\bibitem{AMU} H. Auerbach, S. Mazur, S. Ulam, Sur une propri\'et\'e caract\'eristique de l\'ellipso\"ide, \emph{Monastshe. Math.
	Phys.} {\bf 42} (1935), 45-48.
	
	\bibitem{BR} A. O. Barut, R. R\k{a}czka, Theory of Group Representations and Applications,
	\emph{World Scientific Publishing Company}, 1986, ISBN: 9789813103870, 9813103876.
	
	
\bibitem{Bur}	G. R. Burton. Congruent sections of convex bodies, \emph{Pacific. Journal of Math.}, {\bf 81}, 2, (1979), 303-316.
	
	
\bibitem{BHJM}	G. Bor, L. Hern\'andez Lamoneda, V. Jim\'enez-Desantiago, and L. Montejano, On the isometric conjecture of
	Banach, \emph{Geom. Topol}. {\bf 25} (2021), no. 5, 2621–2642.
	
	
	\bibitem{G}R. J. Gardner, Geometric tomography. Second edition,
	\emph{Encyclopedia of Mathematics and its Applications}, {\bf 58} Cambridge University Press, Cambridge, 2006.
	
	
	
	\bibitem{GMS} I. M. Gelfand, R. A. Minlos, Z. Ya. Shapiro, Representations of the  rotation and  Lorentz groups and their applications, Dover Publications, Mineola, New-York, 2018,
	ISBN: 9780486823850, 0486823857.
	
	
	\bibitem{Gro} H. Groemer, Geometric Applications of Fourier Series and Spherical Harmonics,
	\emph{Encyclopedia of Mathematics and its Applications}, {\bf 61} Cambridge University Press, Cambridge, 2009.
	
	
	
	
	\bibitem{Gr}M. Gromov, On one geometric hypothesis of Banach, [In Russian], \emph{Izv. AN SSSR}, {\bf 31}
	(5) (1967), 1105-1114.


\bibitem{IMN} S.  Ivanov, D. Mamaev, and A. Nordskova, Banach’s isometric subspace problem in dimension
four, \emph{Invent. Math}. {\bf 233} (2023), no. 3, 1393–1425.



\bibitem{K1} A. V. Kuz'minykh, The isoprojection property of the sphere, \emph{Soviet Math. Doklady}, {\bf 14} (1973), 891-895.


\bibitem{K2} A. V. Kuz'minykh, Recovery of a convex body from the set of its projections,  \emph{Siberian Math. J}., {\bf 25} (1984), 284-288.


	\bibitem{M} P. Mani, Fields of planar bodies tangent
	to spheres, \emph{Monatsh. Math.}, {\bf 74} (1970), 145-149.
	
	
	\bibitem{MRS} S. Myroshnychenko, D. Ryabogin, C. Saroglou, Star bodies with completely symmetric sections, \emph{IMRN}, {\bf 10} 2019, 3015-3031.
	
	
	
	
	\bibitem{R}D. Ryabogin, On the continual Rubik's cube, \emph{Adv. math.}, {\bf 231} (2012), 3429-3444.

	\bibitem{Sch1}R. Schneider, Convex bodies with congruent sections, \emph{Bull. London Math. Soc.},
	{\bf 12} (1980), 52-54.
	
	\bibitem{Sch2}R. Schneider, Spherical Harmonics, \emph{Lecture at Kent State University}, 2008.

	
	\bibitem{Sc}R. Schneider, Convex Bodies: The Brunn-Minkowski
	theory, 2nd Edition.
	
	\bibitem{S-W}R. Schneider and W. Weil, Stochastic and Integral Geometry, \emph{Probability and Its Applications}, Springer, 2008.
	
	\bibitem{V}N. Ya. Vilenkin, Special functions and the theory of group representations,  Providence, American Mathematical Society, 1968.
	
	\bibitem{W}W. Weil, Kontinuierliche Linearkombination von
	Strecken, \emph{Math. Z.}, $\textbf{148}$ (1976), 71-84.
	
	
	\bibitem{Za}B. Zawalski, On star-convex bodies with rotationally invariant sections, \emph{Beitr. Algebra Geom.} {\bf 65} (2024), no.3, 495-509.
	
	
	
	
	
\end{thebibliography}
\end{document}